\documentclass[a4paper]{amsart} %feb-28-2023
\usepackage{amsmath,amssymb,amsthm}
\usepackage[backref=page, colorlinks, linkcolor=blue]{hyperref}
\usepackage[T1]{fontenc}
\usepackage{enumitem}

\newtheorem{theorem}{Theorem}[section]
\newtheorem{lemma}[theorem]{Lemma}

\newtheorem{proposition}[theorem]{Proposition}

\theoremstyle{definition}
\newtheorem{definition}[theorem]{Definition}

\newtheorem{remark}[theorem]{Remark}

\begin{document}
\title[Uniform distribution mod $1$ for continued fractions]{Uniform distribution mod $1$ for sequences of ergodic  sums and continued fractions}
\author{Albert M. Fisher}
\address{Instituto de Matemática e Estatística, Universidade de São Paulo}
\email{afisher@ime.usp.br}
\thanks{AF was partially supported by FAPESP grant \#2016/25053-8}
\author{Xuan Zhang}
\address{Instituto de Matemática e Estatística, Universidade Federal Fluminense}
\email{xuanz@id.uff.br}
\date{\today}
\thanks{XZ was partially supported by FAPESP grant \#2018/15088-4.}
\subjclass[2010]{11K50, 11K06, 37A50}%, 60G10}

\begin{abstract}
  We establish a coboundary condition for a sequence of ergodic sums (i.e.~Birkhoff partial sums) to be almost surely uniformly distributed mod $1$. Applications are given when the sequence is generated by a Gibbs-Markov map. In particular, we show that for almost every real number, the sequence of denominators of the convergents of its continued fraction expansion satisfies Benford's law.
\end{abstract}
\maketitle

\section{Introduction}
For $x\in (0,1)$ irrational, denote by $[a_1, a_2, \ldots]$ its continued fraction expansion  and write $p_n/q_n$ for  the $n$-th convergent $[a_1, \ldots, a_n]=1/(a_1+\cdots +1/a_n)$. The sequence $(\log q_n(x))_{n\in\mathbb N}$ behaves much like  a sum of i.i.d.~random variables in that it obeys a range of  classical limit theorems (\cite{Philipp1971}), in this note we are interested in uniform distribution  (u.d.) mod $1$. Recalling the definition, a sequence $(x_n)$ of real numbers is said to be  u.d.~mod $1$ if and only if for every $t\in [0,1)$,
\begin{equation*}
\lim_{N\to\infty}\frac 1N\#\{1\leqslant n\leqslant N: \{ x_n\}\leqslant t\}=t,
\end{equation*}
where $\{x_n\}$ represents the fractional part of $x_n$. Equivalently, on the circle $\mathbb T=\mathbb R/\mathbb Z$ with Lebesgue measure $dt$, then for every continuous function $f\in C(\mathbb T)$ we have:
\[\lim_{N\to\infty}\frac1N \sum_{n=1}^{N}f(x_n)=\int_{\mathbb T} f(t) dt.\]

Uniform distribution mod $1$ is closely related to {\em Benford's law}  or {\em the first-digit law}, that is, the proportion of entries with the first digit $d$, $d\in\{1,\ldots, 9\}$, tends to $\lg (1+1/d)$. Here $\lg=\log_{10}$ is the decimal logarithm. First noted by Benford as well as Newcomb, this specific distribution has been observed to occur in a variety of social science statistics and data sets as well as in many number-theoretical experiments. We refer to \cite{BergerHill2011} for a comprehensive survey of Benford's law. For a nonzero real number $x$, let $s(x)$ be its decimal significand, which is the unique number $t\in [1,10)$ such that $|x|=10^kt$ holds for some $k\in\mathbb Z$. A sequence of real numbers $(x_n)_{n\in\mathbb N}$ is said to be a {\em Benford sequence} if the corresponding sequence of decimal significands  $(s(x_n))_{n\in\mathbb N}$ is distributed in $[1, 10)$ with density $\lg$, that is,
\[\lim_{N\to\infty}\frac1N\#\{1\leqslant n\leqslant N: s(x_n)\leqslant t\}=\lg (t), \quad \forall t\in [1, 10).\]
Many sequences are known to be Benford; examples are $(2^n)_{n\in\mathbb N}$ and the Fibonacci sequence. The sequences of all natural numbers and of the prime numbers are not Benford, but are {\em weakly Benford} (with the usual  Ces\'aro density being replaced by logarithmic density).
Benford's law is related to uniform distribution mod $1$ via the following equivalent condition (see e.g.~\cite[Theorem 4.2]{BergerHill2011}):
\begin{center}$(x_n)_{n\in\mathbb N}$ is a Benford sequence if and only if $(\lg x_n)_{n\in\mathbb N}$ is u.d.~mod $1$.
\end{center}
Hence the question under consideration can be rephrased by whether $(q_n(x))_{n\in\mathbb N}$ is a Benford sequence.

When $x$ is a quadratic irrational, several authors at about  the same time (\cite{JagerLiardet1988, KanemitsuNagasakaRauzyEtAl1988, SchatteNagasaka1991}) showed that $(q_n(x))_{n\in\mathbb N}$ is indeed a Benford sequence.
In particular this reconfirms that the Fibonacci sequence, being equal to $(q_n(\frac{1+\sqrt 5}{2}))_{n\in\mathbb N}$, is Benford. Jager (as mentioned in \cite{KanemitsuNagasakaRauzyEtAl1988}) asked  whether $(q_n(x))_{n\in\mathbb N}$ is Benford for almost every $x$. In this note we answer this question affirmatively.
\begin{theorem}\label{thm:benford}
$(q_n(x))_{n\in\mathbb N}$ is a Benford sequence for almost every $x$. Equivalently, $(\lg q_n(x))_{n\in\mathbb N}$ is u.d.~mod $1$ for almost every $x$.
\end{theorem}
We note that Schatte in \cite{Schatte1990} presented a proof by way of a direct analysis of the continued fractions. Some of his estimates were illustrated via examples,
rather than being given full proofs valid for  the general cases. Here we present a new, detailed proof, based on ergodic theory methods. We believe this more general and hopefully more transparent approach will lead to new insights on this and other  related questions. For one example, it follows immediately from our proof that for any constant $\rho\in\mathbb R$ the sequence $(\lg q_n(x)+ n\rho)_{n\in\mathbb N}$ is also u.d.~mod $1$ for almost every $x$; the same holds for the sequences $(a_1+\cdots + a_n)_{n\in\mathbb N}$ and $(\lg(a_1\cdots a_n))_{n\in\mathbb N}$.

In what follows will use the natural logarithm $\log$ in  place of $\lg$; it will be clear that the proofs work for both.

The almost sure behavior of the sequence $(\log q_n(x))_{n\in\mathbb N}$ is somewhat different from the situation where $x$ is a quadratic irrational, which equivalently has an eventually periodic continued fraction expansion. In this  case, as shown by Jager-Liardet (\cite{JagerLiardet1988}), $\log q_n(x)=\frac{n}{l(x)}\alpha(x)+c_n(x)+o(1)$ for large $n$, where $\alpha(x)$ is an irrational number, $l(x)$ is the period in the continued fraction expansion of $x$, and $c_n=c_{n+l}$ takes a finite number of values periodically. One can then deduce uniform distribution mod $1$ of the sequence $(\log q_n(x))_{n\in\mathbb N}$ from that of the irrational rotation $(n\alpha(x))_{n\in\mathbb N}$, the classical Weyl's Theorem.
In the almost sure case, $\lim_{n\to\infty}\log q_n(x)/n$ converges to the Lévy constant $\frac{\pi^2}{12\log 2}$ and the rate of convergence is of the order $O(\sqrt{n^{-1}\log\log n})$ by the law of iterated logarithm of Philipp-Stackelberg (\cite{PhilippStackelberg1969}). However we do not know if the L\'evy constant is irrational; moreover this rate of convergence is too weak to  show almost sure uniform distribution of $(\log q_n(x))_{n\in \mathbb N}$ as a direct consequence of that for a rotation; an explanation is that this gives  $\limsup$ rather than the sought-for time average information.
Nevertheless, one might still guess that $(\log q_n(x))_{n\in\mathbb N}$ is u.d.~mod $1$ for almost every $x$, given that the statistics of $\log q_n(x)$ resembles well that of a partial sum of i.i.d.~random variables and that, as shown by Robbins (\cite{Robbins1953}),  the sequence of partial sums of i.i.d.~random variables  (if not supported on a rational lattice) is almost surely uniformly distributed (a.s.u.d.) mod $1$.

The main idea of our proof is to approximate the sequence $(\log q_n(x))_{n\in\mathbb N}$ fast enough by a series of sequences of ergodic sums (in other words, sequences of partial sums of stationary processes), each  of which is a.s.u.d.~mod $1$ (Lemma \ref{lem:aprox}). Sequences of partial sums of stationary processes were studied by Holewijn (\cite{Holewijn1969/70}) and more recently by Chenavier, Mass\'e and Schneider (\cite{ChenavierMasseSchneider2018}) where they found a necessary and sufficient condition for such a sequence to be a.s.u.d.~mod $1$. We use a classical argument of Furstenberg to find a condition in terms of coboundaries for sequences of the partial sums generated by ergodic maps (Theorem \ref{thm:asud}). When the sequence is generated by a Gibbs-Markov map we give examples where the coboundary condition is verified (Theorem \ref{thm:ud_a}), and these provide the approximating sequences to $(\log q_n)$.

\section{Sequence of ergodic sums}
Let $(X, \mu)$ be a standard probability space, $T:X\to X $ be a $\mu$-preserving transformation, and $f:X\to \mathbb R$ be a measurable function (also called here an {\em observable}). For any $n\in\mathbb N$, $x\in X$, form the ergodic sum $\mathcal S_n f(x)=f(x)+\cdots + f(T^{n-1}x)$.

Assuming  ergodicity of $T$, we apply an argument of  Furstenberg (\cite[Lemma 2.1]{Furstenberg1961}; see also \cite[Theorem 4.4]{BedfordFisherUrbanski2002}) to show that $\mathcal S_nf(x)$ is almost surely uniformly distributed (a.s.u.d.) mod $1$ under a coboundary condition for $f$. Note that we do not assume the compactness of $X$ as done in \cite{Furstenberg1961}.

\begin{theorem}\label{thm:asud}
Let $T$ be an ergodic $\mu$-preserving transformation and $f:X\to \mathbb R$ measurable. Suppose for all integers $k\neq 0$, the equation $e^{2\pi i k f(x)}=g(x)/g(Tx)$ (a.s.)~has no solution  $g\in L^2(X)$. Then $\mathcal S_nf(x)$ is a.s.u.d.~mod $1$.
\end{theorem}

\begin{proof}
We write $\mathbb T={\mathbb R}/{\mathbb Z}$ with the Lebesgue measure $dt$.
Let $Y=X\times \mathbb T$ and define the skew product  $T_f: Y \to Y$ by $T_f(x,
t)=(Tx, t+f(x))$. Then $T_f^n(x,t)=(T^nx, t+\mathcal S_nf(x))$.
Now firstly
$\lambda=\mu\times dt$ is an invariant measure for $T_f$; this is because for any $F\in L^\infty_\lambda(Y)$
\begin{align*}
\int_Y F\circ T_f d\lambda&=\iint_Y F(Tx, t+f(x))dt d\mu =\iint_Y F(Tx, t)dtd\mu\\
&=\iint_Y F(x, t)dt d\mu=\int_Y F d\lambda
\end{align*}
by the invariance of $dt$ under translation and the invariance of $\mu$ under $T$.
Moreover the skew product transformation is ergodic; we follow  Furstenberg's proof of this part.  If not, then there is a non-constant invariant function $F\in L^2_\lambda(Y)$ for $T_f$, that is, $F\circ T_f=F$ almost everywhere. Now for almost every $x$, $F(x,
\cdot)\in L^2(\mathbb T)$.
We write  its Fourier expansion,  $F(x, t)= \sum_{n\in\mathbb Z}a_n(x)e^{2\pi i n t},$
with $a_n\in L^2_\mu(X)$. Then \[F\circ T_f(x, t) = \sum_{n\in\mathbb Z}a_n(Tx)e^{2\pi i n f(x)}e^{2\pi i n t}.\] $F\circ T_f= F$ implies $a_n(x)=a_n(Tx)e^{2\pi inf(x)}$, hence $|a_n(x)|=|a_n(Tx)|$, for every $n\in\mathbb Z$ and almost every $x$. The ergodicity of $T$ implies that $|a_n(x)|$ is a.s.~a constant  for every $n\in\mathbb Z$. Since $F$ is not  constant, one of $|a_n(x)|$ with $n\neq 0$  is not $0$. Say for $k\neq 0$, $|a_k(x)|$ is a nonzero constant. Then $e^{2\pi i k f(x)}=a_k(x)/a_k(Tx)$ a.s., a contradiction to the hypothesis,
proving ergodicity of $T_f$.

Next we show how u.d.~in $\mathbb T$ follows from this.  Let $h\in C(\mathbb T)$. We extend this to $H: Y\to \mathbb R$ by $H(x,t)= h(t)$. This is a measurable function which is continuous on the fiber over $x$, for each $x$. Note that $H\in L^1_\lambda(Y)$.

The Birkhoff ergodic theorem implies that for every such $H$, hence for each  continuous function $h\in C(\mathbb T)$, there is a full measure set $E_h$ such that for every $(x, t)\in E_h\subseteq Y,$
\begin{equation}\label{eq:1}
\lim_{N\to\infty}\frac1N \sum_{n=0}^{N-1}H\circ T_f^{n}(x,t)=\int_Y
H d\lambda=\int_{\mathbb T}h dt.
\end{equation}
Since $\mathbb T$ is compact, $C(\mathbb T)$ is separable. Suppose $\{h_i\}_{i\in \mathbb N}$ is a dense subset of $C(\mathbb T)$, then \eqref{eq:1} holds for every $h_i$ and all $(x,t)$ in $\cap_{i\in \mathbb N} E_{h_i}$. This passes over to every $h\in C(\mathbb T)$:  we approximate $h$ $\varepsilon$-uniformly by functions $h_i$, and  both sides of \eqref{eq:1} are within $\varepsilon$.

In conclusion,  $\lambda$-almost every $(x,t)$ is a {\em fiber generic point}, that is, \eqref{eq:1} holds for every $h\in C(\mathbb T)$.

Suppose $(x,t)$ is a fiber generic point, then $(x, t+s)$ is also a fiber generic point for any $s\in \mathbb T$. This is because, for every $h\in C(\mathbb T)$ defining $h_s(\cdot)=h(\cdot+s)\in C(\mathbb T)$ and $H_s(x,t)=h_s(t)$,
\begin{equation*}
\lim_{N\to\infty}\frac1N \sum_{n=0}^{N-1}H\circ
  T_f^{n}(x, t+s)=\lim_{N\to\infty}\frac1N
                      \sum_{n=0}^{N-1}H_s\circ
                      T_f^{n}(x,t)=\int_{\mathbb T} h_sdt=\int_{\mathbb T} h d t.
\end{equation*}
So the set of all fiber generic points is $E\times \mathbb T$ for some $E\subset X$. Since this set has full $\lambda$-measure, $\mu(E)=1$. Now for every $x\in E$ $(x,0)$ is a fiber generic point. This means that for every  $h\in  C(\mathbb T)$,
 \[\lim_{N\to\infty}\frac1N \sum_{n=0}^{N-1}h(\mathcal S_nf(x))=\int_{\mathbb T} hd t,\]
%Using continuous functions $h$ to approximate $ \mathbf 1_{[a,b)}$, one then gets that
%$$\lim_{N\to\infty}\frac1N \sum_{n=0}^{N-1}\mathbf 1_{[a,b)}(\{S_nf(x)\})=b-a,$$
 equivalently $\mathcal S_nf(x)$ is u.d.~mod $1$ for every $x\in E$.
\end{proof}

The coboundary condition is close to being a necessary and sufficient condition.
Recall Weyl's criterion (e.g.~\cite{KuipersNiederreiter1974}): $(x_n)$ is u.d. mod $1$ if and only if for every integer $k\neq 0$
\[\lim_{N\to\infty}\frac1N\sum_{n=1}^N e^{2\pi i k x_n}=0.\]

\begin{proposition}\label{prop:ud} Let $T$ be an ergodic $\mu$-preserving transformation
and $f:X\to \mathbb R$ measurable.  Then  if $\mathcal S_nf(x)$ is a.s.u.d. mod $1$,  the equation $e^{2\pi i k f(x)}=g(x)/g(Tx)$ (a.s.) has no solution $g\in L^1$ with $\int g d\mu \neq0$, for each integer $k\neq 0$.
\end{proposition}
\begin{proof}
Suppose for some integer $k\neq 0$ and $g\in L^1$, $e^{2\pi i k f(x)}=g(x)/g(Tx)$ a.s. Then $|g|=|g\circ T|$. We can assume that $|g|$=1 by ergodicity of $T$. So \[\frac1N \sum_{n=1}^N  e^{2\pi i k \mathcal S_nf(x)}  =\frac 1N\sum_{n=1}^N  g(x)\bar{g}(T^nx) \to g(x) \int \bar{g} d\mu\ \text{a.s.}\]
by ergodic theorem. It follows from Weyl's criterion that if $\mathcal S_nf(x)$ is a.s.u.d. mod $1$ then $\int g d\mu =0$.
\end{proof}

For circle-valued skew products the coboundary condition is called {\em aperiodicity} in \cite{AaronsonDenker2001}. This terminology makes sense in view of \cite[Proposition 4.3]{BedfordFisherUrbanski2002}. However that fact for the skew product does not directly imply  the uniform distribution, which is why we give this argument.

\subsection{Gibbs-Markov maps} When $T$ is a Gibbs-Markov map, we give some sufficient conditions for the coboundary condition to hold that are easier to verify. First recall the definition of a Gibbs-Markov map from \cite{AaronsonDenkerUrbanski1993, AaronsonDenker2001}. Let $(\Omega,\mathcal B, \mu)$ be a probability space and let $T$ be a non-singular transformation as defined in ergodic theory; this weakens the usual condition of measure-preserving, to being measure-class-preserving, i.e.~preserving sets of measure zero: that is for $A\in\mathcal B$, $\mu(T^{-1}A)=0$ if and only if $\mu(A)=0$. In the setting of Aaronson, Denker and Urbanski, one generally studies maps which are non-invertible with a countable number of branches. The map restricted to each branch is non-singular while it may or may not be measure-preserving when taken as a whole.

Consider a countable partition $\alpha$ of $\Omega\mod \mu$. Denote by $\alpha_0^{n-1}$ the refined partition $\bigvee_{i=0}^{n-1}T^{-i}\alpha$ and by $\sigma(\cdot)$ the generated $\sigma$-algebra.
\begin{definition}%\label{def:gm}
$T$ is called a Gibbs-Markov map if the following conditions are satisfied.
\begin{enumerate}[label=(\alph*)]
	\item $\alpha$ is a strong generator of $\mathcal{B}$ under $T$, i.e. $\sigma(\{T^{-n}\alpha:n\in\mathbb N\cup\{0\}\})=\mathcal{B} \mod \mu$.
	\item  $\inf\limits_{a\in\alpha}\mu(Ta)>0$.
	\item For every $a\in\alpha$, $Ta\in \sigma(\alpha) \mod \mu$, moreover the restriction $T|_{a}$ is invertible and non-singular.
	\item For every $n\in\mathbb N$ and $a\in\alpha_0^{n-1}$, denote the non-singular inverse branch of $T^{-n}$ on $T^n a$ by $v_a: T^n a\rightarrow a$. We denote the Radon-Nikodym derivative $\frac{d\mu\circ v_a}{d\mu}$ by $v'_a$. There exist $r\in (0,1)$ and $M>0$ such that for any $n\in \mathbb N, a\in\alpha_0^{n-1}$ and $x,y\in T^n a$ a.e.
\[\left|\dfrac{v'_a(x)}{v'_a(y)}-1\right|\leqslant M\cdot r(x,y),\]
\end{enumerate}
where $r(x,y)$ is the metric $r^{m(x,y)}$ with \[m(x,y)=\min\{n\in\mathbb N: T^{n-1}(x) \text{ and } T^{n-1}(y) \text{ belong to different elements of } \alpha \}.\]
\end{definition}
$T$ is said to be topologically mixing if for any $a, b\in\alpha$ there is $n_{a,b}\in\mathbb N$ such that for any $n\geqslant n_{a,b}$, $T^n a \supset b$.  Note that a topologically mixing and measure preserving Gibbs-Markov map is always ergodic (\cite[Corollary]{AaronsonDenker2001}). A function $f$ is said to be locally Lipschitz if \[D_\alpha (f)=\sup_{a\in\alpha}\sup_{x,y\in a}\frac{|f(x)-f(y)|}{r(x,y)}<\infty,\]
and Lipschitz if \[D(f)=\sup_{x,y\in \Omega}\frac{|f(x)-f(y)|}{r(x,y)}<\infty.\]
Let $\beta$ be the finest partition such that $\sigma(T\alpha)=\sigma(\beta)$.

\begin{proposition}\label{prop:cob} Let $T$ be a topologically mixing and $\mu$-preserving Gibbs-Markov map.  Suppose $e^{2\pi i k f(x)}=g(x)/g(Tx)$ a.s.~for some integer $k\neq 0$ and measurable $g$.
\begin{enumerate}
\item If $f$ is $\alpha$-measurable, then $g$ is $\beta$-measurable.
\item Suppose $p$ is a fixed point of $T$ and, with respect to the metric $r$, $p$ lies in the interior of a partition member and every open ball around $p$ has positive measure. If $p$ is a continuity point of both $T$ and $f$ and if $f$ is locally Lipschitz, then $e^{2\pi i k f(p)}=1$.
\end{enumerate}
\end{proposition}
\begin{proof}
\begin{enumerate}
\item This is \cite[Theorem 3.1]{AaronsonDenker2001}.
\item We can assume that $|g|$=1 by ergodicity of $T$. According to \cite[Corollary 2.2]{AaronsonDenker2001}, the condition that $f$ is locally Lipschitz implies that $g$ is Lipschitz. %The ergodicity of $T$ implies that $|g|$ is a constant, say $|g|=1$. Then $g/g\circ T=g\cdot \bar{g}\circ T$ is also Lipschitz, as $$|g(x)\bar{g}(Tx)-g(y)\bar{g}(Ty)|\leqslant |g(x)-g(y)|+|g(Tx)-g(Ty)|\leqslant (M+ M/r)\cdot r(x,y),$$ where $M=D(g)$.
For a fixed point $p$ in the assumption, there is a sequence $x_n$ tending to $p$ such that $|g(x_n)-g(Tx_n)|\leqslant D(g) r(x_n, Tx_n)$ and $e^{2\pi i k f(x_n)}=g(x_n)/g(Tx_n)$. Since both $Tx_n$ and $x_n$ tend to $p$ and since $f$ is continuous at $p$, it follows that $e^{2\pi i k f(p)}=1$.
\end{enumerate}
\end{proof}

\section{Applications to continued fractions}
We return to continued fractions and prove Theorem \ref{thm:benford} in this section.
Let $\Omega=(0,1)$, $\mu$ be the Gauss measure $d\mu=\frac{1}{\log 2}\frac{1}{1+x} dx$,  and $T$ be the continued fraction map $x\mapsto \{1/x\}$. Denoting the continued fraction expansion of an irrational number $x$ by $[a_1(x), a_2(x), \ldots]$, then $T$ is equal to the shift map $[a_1, a_2,\ldots]\mapsto [a_2, a_3, \ldots]$. Let $\alpha$ be the partition into cylinder sets determined by $a_1$, i.e. $\alpha=\{\{x: a_1(x)=n\}: n\in \mathbb N\}$. Let $r(x,y)={(1/2)}^{\max\{n:\,a_j(x)=a_j(y), 1\leqslant j\leqslant n\}}$. The constant $r=1/2$ is chosen so that $r^2\cdot |(T^2)'|\geqslant 1$. One can check that $T$ is a topologically mixing, $\mu$-preserving Gibbs-Markov map (\cite[Example 2]{AaronsonDenker2001}), so we can apply results from the previous section to the continued fraction map.

\begin{proposition}\label{prop:cob_cf}
The equation $e^{2\pi i k f(x)}=g(x)/g(Tx)$ has no solution for any integer $k\neq 0$ and measurable $g$ in each of the following cases.
\begin{enumerate}
\item $f(x)=a_1(x)$
\item $f(x)=\log (a_1(x))$
%\item $f(x)=-\log x$.
\item $f(x)=-\log (T^lx)$, where $l\in\mathbb N\cup\{0\}$.
%\item $f(x)=-\log x+\rho$, where $\rho\in \mathbb R$.
\end{enumerate}
\end{proposition}
\begin{proof}
This follows immediately form Proposition \ref{prop:cob}, since
\begin{enumerate}
\item $a_1(x)$ is $\alpha$-measurable and $\beta=\{\Omega\}$. For $g$ to be $\beta$-measurable, it has to be constant, but $e^{2\pi i k a_1}$ is not.
\item The same reason as item 1.
\item First let $l=0$. Note that $\log x$ is locally Lipschitz, because $D_{[n]}(\log)$ has the order of $\frac 1n$. Consider the fixed point $p=\frac{\sqrt{5}+1}{2}=[1,1,\ldots]$, which is in the interior of the cylinder set $\{x: a_1(x)=1\}$. Both $T$ and $\log$ are continuous at $p$, but $e^{-2\pi i k \log(p)}\neq 1$ for any integer $k\neq 0$. For $l\in\mathbb N$,
$\log\circ T^l$ is locally Lipschitz, as $D_{[n]}(\log\circ T^l)$ has the order of $\frac{1}{r^l n}$. But again $e^{-2\pi i k \log(T^lp)}\neq 1$ for any integer $k\neq 0$.
%\item $e^{2\pi i k (-\log(x)+\rho)}$ cannot be $1$ on both $[1,1,\ldots]$ and $[2,2,\ldots]$.
\end{enumerate}
\end{proof}

Consequently, applying Theorem \ref{thm:asud},
\begin{theorem}\label{thm:ud_a} With respect to $\mu$, hence also with respect to  Lebesgue measure, each of the following sequences is a.s.u.d.~mod $1$.
\begin{enumerate}[label=(\roman*)]
\item $(a_1(x)+\cdots + a_n(x))_{n\in\mathbb N}$
\item
$(\log(a_1(x)\cdots a_n(x)))_{n\in\mathbb N}$
%\item $(-\mathcal S_n \log(x))_{n\in\mathbb N}$ is a.s.u.d.~mod $1$.
\item $(-\mathcal S_n \log(T^l x))_{n\in\mathbb N}$, where $l\in\mathbb N\cup\{0\}$.
\end{enumerate}
\end{theorem}
%\begin{proof}
%Apply Corollary \ref{cor:ud_gm} with $h_1(x)=\log a_1(x)$ and $h_2(x)=-\log(x)$ respectively for each statement. Note that both $h_1$ and $h_2$ are unbounded and that $D_{[n]} h_1=0$ and $D_{[n]} h_2$ is of the order of $1/n$.%$\lg(1+\frac 1n)$.
%\end{proof}
\begin{remark}
Item (ii) implies that $(a_1(x)\cdots a_n(x))_{n\in\mathbb N}$ is a Benford sequence a.s.
\end{remark}

Now we prepare for the proof of Theorem  \ref{thm:benford}. Recalling that $q_n$ denotes the denominator of the convergent $[a_1,\ldots, a_n]$,  to extend the previous theorem to sequences such as $(\log q_n)$ that are not necessarily ergodic sums, we make use of a simple approximation lemma. Given a sequence $(u_n(x))_{n\in\mathbb N}$ and a series of sequences $(u_n^{(k)}(x))_{n\in\mathbb N}$, where $k\in\mathbb N$ and $x\in\Omega$,
\begin{lemma}\label{lem:aprox}
Suppose that for every $k\in\mathbb N$ the sequence $(u_n^{(k)}(x))_{n\in\mathbb N}$ is a.s.u.d.~mod $1$ and that \[\lim_{k\to\infty}\sup_{n> k}\|u_n^{(k)}-u_n\|_\infty=0,\]
then $(u_n(x))_{n\in\mathbb N}$ is also a.s.u.d.~mod $1$.
\end{lemma}
\begin{proof}
For every $m\in\mathbb N$, choose $k=k(m)$ such that \[2\pi\sup_{n> k}\|u_n^{(k)}-u_n\|_\infty<1/m.\]
So there is a full-measure set $\Omega_m'$ such that for every $x\in\Omega_m'$,
\[2\pi\sup_{n> k}|u_n^{(k)}(x)-u_n(x)|<1/m.\]
Again by assumption, there is a full-measure set $\Omega_m$ such that for every $x\in\Omega_m$, $(u_n^{(k(m))}(x))_{n\in\mathbb N}$ is u.d.~mod $1$. Put $\Omega_0=\cap_{m\in\mathbb N} \Omega_m'\cap\Omega_m$, which is a full-measure set. Let $x\in\Omega_0$. Since $(u^{(k(m))}_n(x))_{n\in\mathbb N}$ is u.d.~mod $1$, in view of Weyl's criterion, choose $N_0>2km$ such that for all $N\geqslant N_0$
\[\frac 1N\left|\sum_{n=1}^N e^{2\pi i u^{(k)}_n(x)}\right|\leqslant 1/m.\]
So for all $N\geqslant N_0$
\begin{align*}
\frac1N \left|\sum_{n=1}^N e^{2\pi i u_n(x)}\right| & \leqslant \frac1N\left|\sum_{n=1}^N e^{2\pi i u^{(k)}_n(x)}\right|+\frac1N\sum_{n=1}^N \left|e^{2\pi i u^{(k)}_n(x)}-e^{2\pi i u_n(x)}\right|\\
&\leqslant \frac1N\left|\sum_{n=1}^N e^{2\pi i u^{(k)}_n(x)}\right|+\frac{2k}{N}+2\pi\sup_{n>k}|u_n^{(k)}(x)- u_n(x)|\\
&\leqslant 3/m.
\end{align*}
Hence $\lim_{N\to\infty}\frac 1N\sum_{n=1}^N e^{2\pi i u_n(x)}=0$, and $(u_n(x))_{n\in\mathbb N}$ is u.d.~mod $1$.
\end{proof}

Now consider the difference $\delta_n(x)=\log q_n(x)-(-\mathcal S_n\log(x))$. It is well-known that $\delta_n(x)$ is uniformly bounded; in fact (see e.g.~\cite{Khinchin1997}),

\begin{lemma}\label{lem:delta}
\[\delta_n(x)=-\log(1+T^{n}x\cdot [a_{n},\ldots, a_1])\]
where $[a_n,\ldots, a_1]$ represents the finite continued fraction $1/(a_n+\cdots +1/a_1).$
%Especially $|\delta_n(x)|\leqslant \lg 2.$
\end{lemma}
\begin{proof}
Because \[x=\frac{p_{n-1}r_n+p_{n-2}}{q_{n-1}r_n+q_{n-2}},\] where $r_n=1/(T^{n-1}x)$,
one finds \[T^nx=\frac{xq_n-p_n}{p_{n-1}-xq_{n-1}}.\]
So
\[q_n\cdot x\cdots T^{n-1}x=(-1)^{n-1}q_n(xq_{n-1}-p_{n-1}).\]
Because $p_nq_{n-1}-p_{n-1}q_n=(-1)^{n-1}$ and $\frac{q_{n-1}}{q_n}=[a_n, \ldots, a_1]$, it follows that
\begin{align*}
\delta_n
&=\log(q_n\cdot x\cdots T^{n-1}x)\\
%&=\lg(q_n\langle xq_{n-1}\rangle)\\
&=-\log\left(1+\frac{x-\frac{p_n}{q_n}}{\frac{p_{n-1}}{q_{n-1}}-x}\right)\\
&=-\log(1+T^{n}x\cdot [a_{n},\ldots, a_1]).
\end{align*}
\end{proof}

Moreover, we can approximate $(\delta_n(x))_{n\in\mathbb N}$ by a dynamical sequence up to a uniformly small error.
\begin{lemma}\label{lem:aprox_delta} For all $k\in\mathbb N$, $n>k$ and $x\in (0,1)$,
\[|\delta_k(T^{n-k}x)-\delta_n(x)|<-\log(1-2^{-k/2}).\]
\end{lemma}
\begin{proof}
Fix $n\in\mathbb N$. For every $1\leqslant k\leqslant n$, let ${\tilde p_k}/{\tilde q_k}=[a_n,\ldots, a_{n-k+1}]$ be the $k$-th convergent of the finite continued fraction $[a_n,\ldots, a_1]$, $\tilde r_k=1/[a_{n-k+1}, \ldots, a_1]$ be the $k$-th remainder and let $\tilde p_0=0, \tilde q_0=1$, then
\[[a_n, \ldots, a_1]=\frac{\tilde p_n}{\tilde q_n}=\frac{\tilde p_k\tilde r_{k+1}+\tilde p_{k-1}}{\tilde q_k\tilde r_{k+1}+\tilde q_{k-1}}, \quad 1\leqslant k<n.\]
One calculates
\begin{align*}
&\quad \delta_k(T^{n-k}x)-\delta_n(x)\\
&=-\log(1+T^{n}x\cdot [a_{n},\ldots, a_{n-k+1}])+\log(1+T^{n}x\cdot [a_{n},\ldots, a_1])\\
&=-\log\left(1+\frac{T^nx\cdot([a_n,\ldots, a_{n-k+1}]-[a_n, \ldots, a_1])}{1+T^nx\cdot [a_n, \ldots, a_1]}\right)\\
&=-\log\left(1+\frac{T^nx\cdot \left(\tilde p_k\tilde q_{k-1}-\tilde p_{k-1}\tilde q_k\right)}{\tilde q_k\left(\tilde q_k\tilde r_{k+1}+\tilde q_{k-1}\right)+T^nx\cdot \tilde q_k\left(\tilde p_k\tilde r_{k+1}+\tilde p_{k-1}\right)}\right)\\
&=-\log\left(1+\frac{T^nx\cdot (-1)^{k-1}}{\tilde q_k\left(\tilde q_k\tilde r_{k+1}+\tilde q_{k-1}\right)+T^nx\cdot \tilde q_k\left(\tilde p_k\tilde r_{k+1}+\tilde p_{k-1}\right)}\right),
\end{align*}
hence the lemma follows from the estimate
\begin{align*}
&\quad \left|\frac{T^nx\cdot (-1)^{k-1}}{\tilde q_k\left(\tilde q_k\tilde r_{k+1}+\tilde q_{k-1}\right)+T^nx\cdot \tilde q_k\left(\tilde p_k\tilde r_{k+1}+\tilde p_{k-1}\right)}\right|\\
&< \frac{1}{\tilde q_k\left(\tilde q_k\tilde r_{k+1}+\tilde q_{k-1}\right)}\leqslant \frac{1}{\tilde q_k\left(\tilde q_k a_{n-k}+\tilde q_{k-1}\right)}
=\frac{1}{\tilde q_k\tilde q_{k+1}}\leqslant 2^{-k/2}.
\end{align*}
\end{proof}

\begin{proof}[Proof of Theorem \ref{thm:benford}]
Let $h(x)=-\log(x)$, and for every $k\in\mathbb N$ let \[h^{(k)}=h\circ T^k-\delta_k+\delta_k\circ T.\]  Note that
for all $n> k$, \[\log q_n(x)=\mathcal S_nh+\delta_n=\mathcal S_{n-k}h^{(k)}+\mathcal S_kh+\delta_k-\delta_k\circ T^{n-k}+\delta_n.\]
In order to show that  $(\log q_n(x))_{n\in\mathbb N}$ is a.s.u.d.~mod $1$, applying Lemmas \ref{lem:aprox} and \ref{lem:aprox_delta},  it suffices to show that for every $k$ the sequence \[(\mathcal S_{n-k}h^{(k)}(x)+\mathcal S_k h(x)+\delta_k(x))_{n> k}\] is a.s.u.d.~mod $1$. For every $k$, subtract $\mathcal S_k h(x)+\delta_k(x)$  from every entry of the latter sequence. Since a translation does not change equidistribution, the equidistribution of the sequence $(\mathcal S_{n-k}h^{(k)}(x))_{n> k}=(\mathcal S_nh^{(k)}(x))_{n\in\mathbb N}$ will complete the proof.

According to Proposition \ref{prop:cob_cf}, the equation $e^{2\pi i m h\circ T^k}=g/g\circ T$ has no solution for any integer $m\neq 0$ and measurable $g$, so $e^{2\pi i m h^{(k)}}=g/g\circ T$ also has no solution. Thus,  by Theorem \ref{thm:asud},  $(\mathcal S_nh^{(k)}(x))_{n\in\mathbb N}$ is a.s.u.d.~mod $1$.
\end{proof}

\begin{remark} Lemma \ref{lem:aprox_delta} may be of independent interest. With $T$ being $\mu$-preserving, Lemma \ref{lem:aprox_delta} implies that $\lim_{n\to\infty}\int \delta_n(x) d\mu$ exists and together with the ergodic theorem this implies that for almost every $x$ \[\lim_{n\to\infty}\frac1n\sum_{j=1}^n{\delta_j}(x)=\lim_{n\to\infty}\int \delta_n d\mu.\]
The limit on the left hand side can be explicitly calculated. Let $t_n(x)=T^nx\cdot \frac{q_{n-1}(x)}{q_{n}(x)}$, then Bosma, Jager and Wiedijk (\cite[Theorem 3]{BosmaJagerWiedijk1983}) found that for a.e.~$x$,
$\frac{1}{n}\sum_{j=1}^n\mathbf 1_{(0,z]}(t_j(x))$ converges to $F(z)$ at every $z\in[0,1]$, where
\[F(z)=\frac{1}{\log 2}\left(\log(1+z)-\frac{z}{1+z}\log z\right).\]
Since $\delta_n(x)=-\log(1+t_n(x))$ and since $\log(1+z)$ is a bounded and continuous function on $[0,1]$, one has that for almost every $x$
\[\lim_{n\to\infty}\frac1n\sum_{j=1}^n{\delta_j}(x)=\int_0^1-\log(1+z)dF(z)=-1-\frac12\log 2+\frac{\pi^2}{12\log2}.\]
So $\lim_{n\to\infty}\int \delta_n d\mu$ is equal to this number as well.

Let \[\theta_n(x)=q_n|xq_n-p_n|.\] Since $\delta_n=\log\theta_{n-1}+\log \frac{q_n}{q_{n-1}}$, it follows that for almost every $x$, \[\lim_{n\to\infty}\frac1n\sum_{j=1}^n\log {\theta_j}(x)=-1-\frac12\log 2.\]
This gives an alternative proof of a result of Haas (\cite{Haas2005}).
\end{remark}

%\bibliography{bib0}
\bibliographystyle{alpha}

%\end{document}

\end{document}